\documentclass[12pt]{amsart}
\usepackage{amsfonts}
\usepackage{amssymb, amscd}
\usepackage{amsmath}
\input xy
\xyoption {all}

\usepackage{amsmath}
\usepackage{amsthm}
\usepackage{amssymb}
\usepackage{latexsym}
\usepackage[dvips]{graphics}

\usepackage{geometry}
\newtheorem{thm}{Theorem}[section] 
\newtheorem{cor}[thm]{Corollary}
\newtheorem{defn}[thm]{Definition}
\newtheorem{example}[thm]{Example}

\newtheorem{lemma}[thm]{Lemma}
\newtheorem{prop}[thm]{Proposition}
\newtheorem{remark}[thm]{Remark}

\newtheorem{prob}[thm]{Problem}

\numberwithin{equation}{section}

\newcommand{\LL}{{\mathcal L}}

\newcommand{\OO}{{\mathcal O}}
\newcommand{\HC}{{\mathcal H}}
\newcommand{\VV}{{\mathcal V}}

\newcommand{\MC}{{\mathcal M}}
\newcommand{\DD}{{\mathcal D}}

\newcommand{\Z}{\mathbb{Z}}
\newcommand{\Lef}{\mathbb{L}}

\newcommand{\Q}{\mathbb{Q}}

\newcommand{\C}{\mathbb{C}}

\newcommand{\mh}{\mbox{MHM}}
\newcommand{\hp}{\mbox{MH}}

\newcommand{\mhs}{\mbox{mHs}}

\begin{document}

\title{Characteristic classes of complex hypersurfaces}

\author[Sylvain E. Cappell ]{Sylvain E. Cappell}
\address{S. E. Cappell: Courant Institute, New York University, 251 Mercer Street, New York, NY 10012, USA}
\email {cappell@cims.nyu.edu}

\author[L. Maxim ]{Laurentiu Maxim}
\address{L. Maxim : Department of Mathematics, University of Wisconsin, Madison, 480 Lincoln Dr, Madison WI 53706-1388, USA}
\email {maxim@math.wisc.edu}

\author[J. Sch\"urmann ]{J\"org Sch\"urmann}
\address{J.  Sch\"urmann : Mathematische Institut,
          Universit\"at M\"unster,
          Einsteinstr. 62, 48149 M\"unster,
          Germany.}
\email {jschuerm@math.uni-muenster.de}

\author[Julius L. Shaneson ]{Julius L. Shaneson}
\address{J. L. Shaneson: Department of Mathematics, University of Pennsylvania, 209 S 33rd St., Philadelphia, PA 19104, USA}
\email {shaneson@sas.upenn.edu}



\begin{abstract} The Milnor-Hirzebruch class of a locally complete intersection $X$ in an algebraic manifold $M$ measures the difference between the (Poincar\'e dual of the) Hirzebruch class of the virtual tangent bundle of $X$ and, respectively, the Brasselet-Sch\"urmann-Yokura (homology) Hirzebruch class of $X$. In this note, we calculate the Milnor-Hirzebruch class of a globally defined algebraic hypersurface $X$ in terms of the corresponding Hirzebruch invariants of singular strata in a Whitney stratification of $X$. Our approach is based on Sch\"urmann's specialization property for the motivic Hirzebruch class transformation of Brasselet-Sch\"urmann-Yokura. The present results also yield calculations of Todd,  Chern and $L$-type characteristic classes of hypersurfaces.
\end{abstract}

\maketitle


\section{Introduction}

An old problem in geometry and topology is the computation of topological and analytical invariants of complex  hypersurfaces, such as  Betti numbers,  Euler characteristic, signature, Hodge numbers and Hodge polynomial, etc.; e.g., see \cite{D,H,KW,Li}. While the non-singular case is easier to deal with, the singular setting requires a subtle analysis of the relation between the local and global topological and/or analytical structure of singularities. For example,  the Euler characteristic of a smooth projective hypersurface  depends only on its degree and dimension. More generally, Hirzebruch \cite{H} showed that the Hodge polynomial of smooth hypersurfaces has a simple expression in terms of the degree and the cohomology class of a hyperplane section. However, in the singular context the invariants of a hypersurface inherit additional contributions from the singular locus.  For instance, the Euler characteristic of a projective hypersurface  with only isolated singularities differs (up to a sign) from that of a smooth hypersurface by the sum of Milnor numbers associated to the singular points.  
In \cite{CLMS}, the authors studied the Hodge theory of one-parameter degenerations of smooth  compact hypersurfaces, where the aim was to compare the Hodge polynomials of the general (smooth) fiber and respectively  special (singular) fiber of such a family of hypersurfaces. By using Hodge-theoretic aspects of the nearby and vanishing cycles \cite{De2,Sa1} associated to the family, the authors obtained in \cite{CLMS} a formula expressing the difference of the two  polynomials in terms of invariants of singularities of the special fiber (see also \cite{Di} for the corresponding treatment of Euler characteristics). 

In this note we study the (homology) Hirzebruch  classes \cite{BSY} of singular hypersurfaces, and derive characteristic class versions of the above-mentioned results from \cite{CLMS}. As these parametrized families of classes include at special values versions (known in many special cases to be the standard ones) of Todd-classes, Chern-classes and $L$-classes, the results described in this paper yield new formulae for all of these. We obtain results both for intersection homology based versions of such classes, as well as for standard homology based versions of them. These, of course, are equal for smooth varieties, but in general differ. 
Formulae for such characteristic classes in the settings of stratified submersions were obtained by some of the present authors in \cite{CMS1,CMS2}. Here by combining results and methods of those papers with a recent result of the fourth author \cite{Sc}, we in particular obtain results which are the counterpart for divisors and, more generally, for regular embeddings to the above-mentioned submersion results. By using the good fit between the results of \cite{CMS2} with that of \cite{Sc}, and where details paralleled those of our earlier papers just giving indications, we are able to give succinct proofs.
The present results on embedding have independent interest, e.g., because of their relation to knot-theoretic invariants and their generalizations in the singular setting, see  \cite{CS0,ShCa,Li94,Max05}.

\bigskip

The study in this note can be done in the following general framework: Let $X \overset{i}{\hookrightarrow} M$ be the inclusion of an algebraic hypersurface $X$  in a  complex algebraic manifold $M$. If $N_XM$ denotes the normal bundle of $X$ in $M$, then the {\it virtual tangent bundle} of $X$, that is, 
\begin{equation}T_{{\rm vir}}X:=[i^*TM-N_XM] \in K^0(X),\end{equation} is independent of the embedding in $M$ (e.g., see \cite{Fu}[Ex.4.2.6]), so it is a well-defined element in the Grothendieck group of vector bundles on $X$. Of course 
$$T_{{\rm vir}}X=[TX] \in K^0(X),$$
in case $X$ is a smooth algebraic submanifold.
If $cl^*$ denotes a cohomology characteristic class theory, then one can associate to the pair $(M,X)$ an {\it intrinsic}  homology class (i.e., independent of the embedding $X\hookrightarrow M$) defined as:
\begin{equation}cl_*^{\rm vir}(X):=cl^*(T_{{\rm vir}}X) \cap [X] \in H_*(X)\:.\end{equation} 
Here $[X]\in H_*(X)$ is the {\it fundamental class} of $X$ in a suitable homology theory  $H_*(X)$ (such as  Borel-Moore homology $H_{2*}^{BM}(X)$ or Chow groups $CH_*(X)$ with integer or rational coefficients).\\

Assume, moreover, that there is a homology characteristic class theory $cl_*(-)$ for complex algebraic varieties, functorial for proper morphisms, obeying the rule that for $X$ smooth  $cl_*(X)$ is the Poincar\'e dual of $cl^*(TX)$.  If $X$ is smooth,  then clearly we have that  $$cl_*^{\rm vir}(X)=cl^*(TX) \cap [X]=cl_*(X) \:.$$ However, if $X$ is singular, the difference between  the homology classes $cl_*^{\rm vir}(X)$ and $cl_*(X)$ depends in general on the singularities of $X$. This motivates the following
\begin{prob} Describe the difference $cl_*^{\rm vir}(X)-cl_*(X)$ in terms of the geometry of the singular locus of $X$.\end{prob}
The strata of the singular locus have a rich geometry, beginning with generalizations of knots which describe their local link pairs. This ``normal data", encoded in algebraic geometric terms via, e.g., the mixed Hodge structures on the (cohomology of the) corresponding Milnor fibers,  will play a fundamental role in our study of characteristic classes of hypersurfaces.

The above problem is usually studied in order to understand the complicated homology classes $cl_*(X)$ in terms of the simpler virtual classes $cl_*^{\rm vir}(X)$ and these difference terms measuring the complexity of singularities of $X$.

There are a few instances in the literature where, for the appropriate choice of $cl^*$ and $cl_*$, this problem has been solved.
For example, if $cl^*=L^*$ is the Hirzebruch $L$-polynomial in the Pontrjagin classes \cite{H}, the difference between the intrinsic homology class $L_*^{\rm vir}(X):=L^*(T_{{\rm vir}}X) \cap [X]$ and the Goresky-MacPherson $L$-class $L_*(X)$ (\cite{GM1}) for $X$ a {\it compact} complex hypersurface was explicitly calculated  in \cite{CS0,ShCa}  as follows: fix a Whitney stratification of $X$ and let $\VV_0$ be the set of  strata $V$ with  ${\rm dim} V<{\rm dim} X$; then if  all $V \in \VV_0$ are assumed simply-connected,
\begin{equation}\label{CS}
L_*^{\rm vir}(X)-L_*(X)=\sum_{V \in \VV_0} \sigma({\rm lk}(V)) \cdot L_*(\bar{V}),
\end{equation} 
where $\sigma({\rm lk}(V)) \in \Z$ is a certain signature invariant associated to the link pair of the stratum $V$ in $(M,X)$. (This result is in fact of topological nature, and holds more generally for a suitable compact stratified pseudomanifold $X$, which is PL-embedded in real codimension two in a manifold $M$; see \cite{CS0,ShCa} for details.) Here the Goresky-MacPherson $L$-class $$L_*(X)= L_*([IC'_X])$$
is the $L$-class of the shifted (self-dual) intersection cohomology complex 
$$IC'_X:=IC_X[-{\rm dim}(X)]$$ of $X$.
(For a functorial $L$-class transformation in the complex algebraic
context compare with \cite{BSY}.)\\

If $cl^*=c^*$ is the total Chern class in cohomology, the problem amounts to comparing the Fulton-Johnson class $c^{FJ}_*(X):=c_*^{vir}(X)$ (e.g., see \cite{Fu,FJ}) with the  homology Chern class $c_*(X)$ of MacPherson \cite{MP}. Here $c_*(X):=c_*(1_X)$, with
$$c_*: F(X)\to H_*(X)$$
the functorial Chern class transformation of MacPherson \cite{MP}, defined on the group
$F(X)$ of complex algebraically constructible functions.
The difference between these two classes is measured by the so-called {\it Milnor class}, $\MC_*(X)$, which is studied in \cite{Alu0,BLSS,BLSS2,BSS,Max,PP,Sch,Sch2,Y}. This is a homology class supported on the singular locus of $X$, and it was computed in \cite{PP} (see also \cite{Sch,Sch2,Y,Max}) as a weighted sum in the Chern-MacPherson classes of closures of singular strata of $X$, the weights depending only on the normal information to the strata. For example, if $X$ has only isolated singularities, the Milnor class equals (up to a sign) the sum of the Milnor numbers attached to the singular points, which also explains the terminology.\\

Lastly, if $cl^*=td^*$ is the Todd class, then the  Verdier-Riemann-Roch theorem
for the Todd-class transformation
$$td_*: G_0(X)\to H_*(X)$$
of Baum-Fulton-MacPherson \cite{BFM, Fu}
 can be used to show that 
 $$td_*^{\rm vir}(X):=td^*(T_{{\rm vir}}X) \cap [X]= td_*(X)$$ 
 equals in fact the Baum-Fulton-MacPherson Todd class $td_*(X):=td_*([\OO_X])$ of $X$; see \cite{BFM}, (IV.1.4). Here $[\OO_X]$ is the class of the structure sheaf in the Grothendieck
 group $G_0(X)$ of coherent algebraic $\OO_X$-sheaves.

\bigskip

A main goal of this note is to study the (unifying) case when $cl^*=T_y^*$ is the (total) cohomology Hirzebruch class of the generalized Hirzebruch-Riemann-Roch theorem \cite{H}.
The aim is to show that the results stated above are part of a more general philosophy, derived from comparing the intrinsic homology class (with polynomial coefficients) \begin{equation}{T_y}^{\rm vir}_*(X):=T_y^*(T_{{\rm vir}}X) \cap [X] \in H_*(X)\otimes \Q[y]\end{equation}  
with the motivic Hirzebruch class ${T_y}_*(X)$ of \cite{BSY}.  This approach is motivated by  the fact  that  the $L$-class $L^*$, the Todd class $td^*$ and resp. the Chern class $c^*$ are all  suitable specializations (for $y=1,0,-1$, respectively) of the Hirzebruch class $T_y^*$; see \cite{H}.  Here ${T_y}_*(X):={T_y}_*([id_X])$, with
$${T_{y}}_*: K_0(var/X) \to H_*(X)\otimes \Q[y]$$
the functorial Hirzebruch class transformation of Brasselet-Sch\"urmann-Yokura \cite{BSY} defined on the relative Grothendieck group $K_0(var/X)$ of complex algebraic varieties over $X$.

In fact, we also use the description ${T_{y}}_*={MHT_{y}}_*\circ \chi_{Hdg}$ in terms of algebraic
mixed Hodge modules, with  
\begin{equation}\label{t}{MHT_y}_*: K_0({\rm MHM}(X)) \to H_*(X) \otimes \Q[y,y^{-1}]\end{equation} 
the corresponding functorial Hirzebruch class transformation of Brasselet-Sch\"urmann-Yokura \cite{BSY, CMS2, Sch3}
defined on the Grothendieck group $K_0({\rm MHM}(X))$ of algebraic mixed Hodge modules on $X$.
Here $\chi_{Hdg}: K_0(var/X) \to K_0({\rm MHM}(X))$ is the natural group homomorphism given by
(e.g., see \cite{Sch3}[Cor.4.10]):
$$[f: Z\to X] \mapsto [f_!\Q^H_Z] \:.$$
Then the homology Hirzebruch class ${T_y}_*(X)={MHT_{y}}_*([\Q^H_X])$ is the value taken on the (class of the) constant Hodge sheaf $\Q^H_X$ by the natural transformation ${MHT_y}_*$, since
$\chi_{Hdg}([id_X])= [\Q^H_X]$.  For $X$ pure-dimensional, the use of mixed Hodge modules  also allows us to consider the Intersection  Hirzebruch class (as in \cite{CMS2, Sch3}):
$${IT_y}_*(X):={MHT_y}_*([IC'^H_X])\in H_*(X)\otimes \Q[y]$$
corresponding to the shifted intersection cohomology Hodge module $IC'^H_X:=IC^H_X[-{\rm dim}(X)]$. 
This is sometimes more natural, especially for the comparison with the $L$-class
$L_*(X)$ of $X$.

\bigskip
Assume in what follows that the complex algebraic variety $X$ is a
{\it hypersurface}, which is globally defined as the zero-set  $X=\{f=0\}$ (of codimension one) of an algebraic function $f:M \to \C$ on a complex algebraic manifold $M$. (But see the discussion in Remark \ref{gen} on generalizing this to local complete intersections, e.g., hypersurfaces without a global equation.)
The main result of this note is the following, where as before,  $H_*(X)$ denotes either the Borel-Moore homology in even degrees $H^{BM}_{2*}(X)$, or the Chow group $CH_*(X)$:
 
\begin{thm}\label{main} Let $\VV$ be a fixed complex algebraic Whitney stratification of $X$, and denote by $\VV_0$ the collection of all singular strata (i.e., strata $V \in \VV$ with ${\rm dim}(V)<{\rm dim} X$). For each $V \in \VV_0$, let $F_v$ be the Milnor fiber of a point $v \in V$.
Assume that all strata $V \in \VV_0$ are simply-connected. Then:
\begin{equation}\label{eq0}
{T_y}^{\rm vir}_*(X) - {T_y}_*(X)
=\sum_{V \in \VV_0}\left( {T_y}_*(\bar V) - {T_y}_*({\bar V} \setminus V) \right) \cdot \chi_y([\tilde{H}^*(F_v;\Q)])\:.
\end{equation} 
If, moreover, for each $V \in \VV_0$, we define inductively $$\widehat{IT}_y(\bar V):= {IT_y}_*(\bar V)- \sum_{W < V}\widehat{IT}_y(\bar W) \cdot \chi_y([IH^*(c^{\circ}L_{W,V})])\:,$$ 
where the summation is over all strata $W \subset {\bar V} \setminus V$ and
$c^{\circ}L_{W,V}$ denotes the open cone on the link of  $W$ in
$\bar{V}$, then:
\begin{equation}\label{eq1}
{T_y}^{\rm vir}_*(X) - {T_y}_*(X)=\sum_{V \in \VV_0} \widehat{IT}_y(\bar V) \cdot \chi_y([\tilde{H}^*(F_v;\Q)])\:.
\end{equation} 
\end{thm}

\begin{remark}\label{new}\rm The assumptions in the first part of the above theorem can be weakened, in the sense that instead of a Whitney stratification we only need a partition of the singular locus $X_{\rm sing}$ into disjoint locally closed complex algebraic submanifolds $V$, such that the restrictions $\Phi_f(\Q_M)\vert_V$ of the vanishing cycle complex to all pieces $V$ of this partition have constant cohomology sheaves (e.g., these are locally constant sheaves on each $V$, and the pieces $V$ are simply-connected).  In particular, the above theorem can be used for computing the Hirzebruch class of the Pfaffian hypersurface and, respectively,  of the Hilbert scheme $$(\C^3)^{[4]}:=\{df_4=0\}\subset M_4$$ considered in \cite{DS}[Sect.2.4 and Sect.3]. Indeed, the singular loci of the two hypersurfaces under discussion have ``adapted" partitions as above with only simply-connected strata (cf. \cite{DS}[Lem.2.4.1 and Cor.3.3.2]). Moreover, the mixed Hodge module corresponding to the vanishing cycles of the defining function, as well as its Hodge-Deligne polynomial are calculated in \cite{DS}[Thm.2.5.1, Thm.2.5.2 Cor.3.3.2 and Thm.3.4.1].
So Thm.\ref{main} above can be used for obtaining class versions of these results from \cite{DS}.
\end{remark}

By the functoriality of ${T_y}_*$ and ${MHT_y}_*$, all homology characteristic classes of closures of strata are regarded in the homology $H_*(X)\otimes \Q[y,y^{-1}]$ of the ambient variety $X$.
Moreover, the cohomology groups $\tilde{H}^k(F_v;\Q)$ and $IH^k(c^{\circ}L_{W,V})$ carry  canonical mixed Hodge structures coming from stalk formulae such as
\begin{equation}\label{stalk}\tilde{H}^k(F_v;\Q)\simeq H^k(\Phi_f(\Q_M)_v),\end{equation}
and the functorial calculus of algebraic mixed Hodge modules (see the formulae (\ref{cone}) and (\ref{eq60}) in Section \ref{sec.comp}). By taking the alternating sum of these cohomology groups  in the Grothendieck group $K_0(\mhs)$ of (rational) mixed Hodge structures, we get classes
$$[\tilde{H}^*(F_v;\Q)], [IH^*(c^{\circ}L_{W,V})] \in K_0(\mhs)\:,$$
to which one can then apply the ring homomorphism (with $F^{\centerdot}$ the Hodge filtration)
$$\chi_y:  K_0(\mhs)\to \Z[y,y^{-1}];\: \chi_y([H]):=\sum_p {\rm dim} Gr^p_F(H \otimes \C) \cdot (-y)^p \:.$$

The requirement in Theorem \ref{main} that all strata in $X$ are simply-connected is only used to assure that all monodromy considerations become trivial to deal with. 
Moreover, in some cases much interesting information is readily available without any monodromy assumptions. 
For example, if $X$ has only {\it isolated} singularities
(so that the monodromy assumptions become vacuous), both formulae of the theorem agree, and the two classes ${T_y}^{\rm vir}_*(X)$ and resp. ${T_y}_*(X)$ coincide except in degree zero, where their difference is measured (up to a sign) by the sum of Hodge polynomials associated to the middle cohomology of the corresponding Milnor fibers attached to the singular points. More precisely, we have in this case that:
\begin{equation}
{T_y}^{\rm vir}_*(X) - {T_y}_*(X)=\sum_{x \in X_{\rm sing}} (-1)^{n} \chi_y([\tilde{H}^n(F_x;\Q)])= \sum_{x\in X_{\rm sing}} \chi_y([\tilde{H}^*(F_x;\Q)]),
\end{equation}
where  $F_x$ is the Milnor fiber of the isolated hypersurface singularity germ $(X,x)$, and $n$ is the complex dimension of $X$. 
The Hodge $\chi_y$-polynomials of the Milnor fibers can in general be computed from the better known {\it Hodge spectrum} of the singularities, and are just Hodge-theoretic refinements of the Milnor numbers, since
$$\chi_{-1}([\tilde{H}^*(F_x;\Q)])= \chi([\tilde{H}^*(F_x;\Q)])$$
is the reduced Euler characteristic of the Milnor fiber $F_x$.
 For this reason, we regard the difference 
 \begin{equation}
 {{\MC}T_y}_*(X):={T_y}^{\rm vir}_*(X) - {T_y}_*(X) \in H_*(X) \otimes \Q[y]
 \end{equation} 
 as a Hodge-theoretic Milnor class, and call it \emph{the Milnor-Hirzebruch class} of the hypersurface $X$. In fact, it is always the case that by substituting $y=-1$ into ${{\MC}T_y}_*(X)$ we obtain the (rationalized) Milnor class of $X$ (this follows from the commutative diagram (\ref{special}) below). Therefore, Theorem \ref{main} specializes in this case to a computation of the rationalized Milnor class of $X$, and the resulting formula holds without any monodromy assumptions (compare \cite{Max}).

\bigskip
The key ingredient used in the proof of Theorem \ref{main} is the specialization property for the motivic Hirzebruch class transformation ${MHT_y}_*$ (see \cite{Sc}). This is a generalization of Verdier's result \cite{V} on the specialization of the MacPherson Chern class transformation, which was used in \cite{PP,Sch,Sch2} for computing the Milnor class of $X$, and shows that the Milnor-Hirzebruch class ${{\MC}T_y}_*(X)$ of $X=f^{-1}(0)$ is entirely determined by the vanishing cycles of $f:M \to \C$ (see Theorem \ref{M}):
\begin{equation}\label{DT}
\MC{T_y}_*(X):={T_y}^{\rm vir}_*(X) - {T_y}_*(X)= T_{y*}(\Phi^m_f([id_M]))=
{MHT_y}_*(\Phi'^H_f(\left[\Q^H_M\right]))\:.
\end{equation}
As we will see below, equation (\ref{DT}) is an enriched version of the (localized) Milnor class formula 
$$\MC_*(X)=c_*(\Phi_f(1_M)) \in H_*(X_{\rm sing}),$$
whose degree appeared recently in the computation of {\it Donaldson-Thomas invariants}, e.g., see \cite{Beh,BBS,DS,JS}. In particular, in [\cite{JS}, Sect.4] the authors express hope that the Donaldson-Thomas theory could be lifted from constructible functions to mixed Hodge modules. We believe our approach is tailored to serve such a purpose. Similarly, motivic nearby and vanishing cycles are used in \cite{BBS}.

Note that $\Phi^m_f([id_M])$ and
$\Phi'^H_f(\left[\Q^H_M\right])$ in equation (\ref{DT}) are supported on the singular locus $X_{\rm sing}$ of $X$. So by the functoriality of the transformations ${T_{y}}_*$ and ${MHT_y}_*$ (for the closed inclusion $X_{\rm sing}\hookrightarrow X$), we can regard
\begin{equation}
\MC{T_y}_*(X)= {T_y}_*(\Phi^m_f([id_M]))=
{MHT_y}_*(\Phi'^H_f(\left[\Q^H_M\right])) \in H_*(X_{\rm sing})\otimes\Q[y]
\end{equation}
as a {\it localized} Milnor-Hirzebruch class.\\

Let us explain more of the notation appearing in equation (\ref{DT}). First note that one can use the nearby- and vanishing cycle functors $\Psi_f$ and $\Phi_f$ either on the motivic level of localized (at the class of the affine line) relative Grothendieck groups $$\MC(var/-):=K_0(var/-)[\Lef^{-1}]$$ (see \cite{Bit,GLM}), or on the Hodge-theoretical level of algebraic mixed Hodge modules 
(\cite{Sa0, Sa1}), ``lifting'' the corresponding functors on the level of  algebraically constructible sheaves 
(\cite{Di, Sch}) and algebraically constructible functions (\cite{Sch,V}), so that the following diagram commutes:
\begin{equation}\label{all-nearby}\begin{CD}
K_0(var/M)\\
@VVV\\
\MC(var/M) @> \Psi_f^m, \Phi_f^m >>  \MC(var/X)\\
@V \chi_{Hdg} VV  @V \chi_{Hdg} VV  \\
K_0(\mh(M)) @> \Psi'^H_f, \Phi'^H_f >>  K_0(\mh(X))\\
@V rat VV @V rat VV \\
K_0(D^b_c(M)) @> \Psi_f, \Phi_f >>  K_0(D^b_c(X)) \\
@V \chi_{stalk} VV @V \chi_{stalk} VV \\
F(M) @> \Psi_f, \Phi_f >> F(X)  \:.
\end{CD}\end{equation}
We also use the notation $\Psi'^H_f:= \Psi^H_f[1]$ and $\Phi'^H_f:=\Phi^H_f[1]$ for the shifted
functors, with $\Psi^H_f, \Phi^H_f: MHM(M)\to MHM(X)$ and $\Psi_f[-1],\Phi_f[-1]: Perv(M)\to Perv(X)$ preserving mixed Hodge modules and perverse sheaves, respectively. 

\begin{remark}\rm
The {\it smoothness} of $M$ is {\it not} used for this commutativity.
Moreover:
\begin{enumerate}
\item The motivic nearby and vanishing cycles functors of \cite{Bit,GLM}
take values in a refined {\em equivariant} localized Grothendieck group
$\MC^{\hat{\mu}}(var/X)$ of equivariant algebraic varieties over $X$ with a ``good'' action
of the pro-finite group $\hat{\mu}=\lim \mu_n$ of {\em roots of unity}
(for the projective system $\mu_{d\cdot n}\to \mu_n: \xi\mapsto \xi^d$). By definition,
this factorizes over a ``good'' action of a finite quotient group $\hat{\mu}\to \mu_n$
of $n$-th roots of unity. 
\item In our applications above we don't need to take this action into account.
So we use the composed horizontal transformations in the following commutative diagram
(see \cite{GLM}[Prop.3.17]):
\begin{equation}\label{comp}\begin{CD}
M_0(var/M) @> \Psi_f^m, \Phi_f^m >> \MC^{\hat{\mu}}(var/X) @> forget >> \MC(var/X)\\
@V \chi_{Hdg} VV  @V \chi_{Hdg} VV  @VV \chi_{Hdg} V\\
K_0(MHM(M)) @> \Psi'^H_f, \Phi'^H_f >> K^{mon}_0(MHM(X)) @> forget >> K_0(MHM(X)) \:.\\
\end{CD}\end{equation}
Here $K^{mon}_0(MHM(X))$ is the Grothendieck group of algebraic mixed Hodge modules
with a finite order automorphism, which in our case is induced from the semi-simple part 
$T_s$ of the monodromy automorphism acting on $\Psi^H_f, \Phi^H_f$.
\item Also note that for the commutativity of diagram (\ref{comp}) one has to use $\Psi'^H_f$ (as opposed to 
$\Psi^H_f$, as stated in \cite{GLM}[Prop.3.17]; this fits in fact with the reference given in
the proof of loc.cit.). Moreover, the Grothendieck group  $\MC^{\hat{\mu}}(var/X)$ used in \cite{GLM}
is finer than the one used in \cite{Bit}. But both definitions of the motivic nearby and vanishing
cycle functors are compatible (\cite{GLM}[Rem.3.13]), and $\chi_{Hdg}$ also factorizes over 
$\MC^{\hat{\mu}}(var/X)$ in the sense of \cite{Bit} by the same argument as for \cite{GLM}[(3.16.2)].
\item In a future work we will define a ``spectral Hirzebruch class transformation''
$${MHT_t}_*: K^{mon}_0(MHM(X)) \to \bigcup_{n\geq 1} H_*(X)\otimes \Q[t^{\frac{1}{n}}, t^{-\frac{1}{n}}]\:,$$
which is a class version of the {\it Hodge spectrum} (e.g., see \cite{GLM})
$$hsp: K^{mon}_0(\mhs) \to \bigcup_{n\geq 1}  \Z[t^{\frac{1}{n}}, t^{-\frac{1}{n}}]\:.$$
These spectral invariants are refined versions (for $t=-y$) of the Hirzebruch class transformation ${MHT_{y}}_*$ and the $\chi_y$-genus, respectively. We will get in particular a refined {\it spectral Milnor-Hirzebruch class}, needed for a
suitable Thom-Sebastiani result. 
\end{enumerate}
\end{remark}

Let $i^!: H_*(M)\to H_{*-1}(X)$ be the homological Gysin transformation (as defined in \cite{V,Fu}).
As already mentioned before, the key ingredient used in the proof of Theorem \ref{main} is the following specialization property for the motivic Hirzebruch class transformation ${MHT_y}_*$:
\begin{thm}(\cite{Sc}) \label{main-sch}
${MHT_y}_*$ commutes with specialization, that is: 
\begin{equation}\label{sp}
{MHT_y}_*(\Psi'^H_f(-))=i^! {MHT_y}_*(-): \:K_0({\rm MHM}(M)) \to
 H_*(X) \otimes \Q[y,y^{-1}] \:.
\end{equation}
\end{thm}
Again the smoothness of $M$ is not needed here, but only the fact that $X=\{f=0\}$ is a global hypersurface (of codimension one).
By the definition of $\Psi_f^m$ in \cite{Bit,GLM}, one has that
$$\Psi_f^m(K_0(var/M))\subset im(K_0(var/X)\to \MC(var/X))\:,$$
so ${T_{y}}_*\circ \Psi_f^m$ maps $K_0(var/M)$ into
$H_*(X)\otimes\Q[y]\subset H_*(X)\otimes\Q[y,y^{-1}]$.
Together with \cite{Sch3}[Prop.5.2.1] one therefore gets the following commutative diagram
of specialisation results:
\begin{equation}\label{special}\begin{CD}
K_0(var/M) @> {T_{y}}_*\circ \Psi_f^m = > i^!\circ {T_{y}}_* > H_*(X)\otimes \Q[y]\\
@V \chi_{Hdg} VV @VVV\\
K_0(\mh(M)) @> {MHT_{y}}_*\circ \Psi'^H_f = > i^!\circ {MHT_{y}}_* > H_*(X)\otimes \Q[y,y^{-1}]\\
@V \chi_{stalk}\circ rat VV @VV y=-1 V \\
F(M) @> c_*\circ \Psi_f = > i^! c_* > H_*(X)\otimes \Q \:.
\end{CD}\end{equation}

\begin{remark}\label{gen}\rm The problem of understanding the class ${{\MC}T_y}_*(-)$ in terms of invariants of the singularities can be formulated in more general contexts, e.g., in the complex analytic setting, for complete intersections or even for regular embeddings of arbitrary codimensions.
And the specialization result of  Theorem \ref{main-sch} can also be used in these cases.
In fact, for {\it global} complete intersections $X=\{f_1=0,\dots,f_k=0\}$ one can iterate this specialization result and get
$${MHT_y}_*(\Psi'^H_{f_1}\circ \cdots \circ \Psi'^H_{f_k}(-))=i^! {MHT_y}_*(-)\:,$$
with $\Psi_{f_1}\circ \cdots \circ \Psi_{f_k}$ related to the Milnor fibration of the
ordered tuple $(f_1,\dots,f_k): M\to \C^k$ in the sense of \cite{MCP}.
And for a general  complete intersection or regular embedding (e.g., for a hypersurface
$X$ without a global equation), one can apply the specialization result
to the so-called ``deformation to the normal cone" (compare \cite{Sch, Sch2} for the case of Milnor-Chern classes). However, for simplicity, we restrict ourselves to the case of globally defined hypersurfaces in complex algebraic manifolds.\end{remark} 

A motivic approach to Milnor-Hirzebruch classes was recently and independently developed by Yokura \cite{Y09}.

\section{Background on Hirzebruch classes of singular varieties}
We assume the reader is familiar with the basics of Saito's theory of algebraic mixed Hodge modules and with the functorial calculus of Grothendieck groups. For a quick survey of these topics see \cite{Sa}, \cite{CMS2}[Sect.3] or \cite{MS}[Sect.2.2-2.3]. We only recall here the construction and main properties of Hirzebruch classes of (possibly singular) complex algebraic varieties, as developed by Brasselet, Sch\"urmann and Yokura in \cite{BSY}.
For the motivic approach in terms of the relative Grothendieck group of complex algebraic varieties
(as indicated in the Introduction) we refer to \cite{BSY}, whereas for the Hodge-theoretical approach used here we refer to the recent overview \cite{Sch3}.\\ 

For any complex algebraic variety $X$, let ${\mh}(X)$ be the abelian category of Saito's algebraic mixed Hodge modules on $X$. For any
$p \in \Z$, M. Saito \cite{Sa1} constructed a functor of triangulated categories
\begin{equation} gr^F_pDR: D^b\mh(X) \to D^b_{coh}(X)\end{equation}
commuting with proper push-down,
with $gr^F_pDR(\MC)=0$ for almost all $p$ and $\MC$ fixed,
where $D^b_{coh}(X)$ is the bounded
derived category of sheaves of $\mathcal{O}_X$-modules with coherent
cohomology sheaves.
If $\Q_X^H \in D^b\mh(X)$ denotes the constant
Hodge module on $X$, and if $X$ is smooth and pure dimensional, then
$gr^F_{-p} DR(\Q_X^H) \simeq \Omega^p_X[-p]$. The transformations
$gr^F_pDR$ induce functors on the level of Grothendieck groups.
Therefore, if $$G_0(X) \simeq K_0(D^b_{coh}(X))$$ denotes the
Grothendieck group of coherent sheaves on $X$, we get a group
homomorphism (the {\em motivic Chern class transformation})
\begin{equation}\label{grF}
MHC_y: K_0(\mh(X)) \to G_0(X) \otimes \Z[y, y^{-1}]\;;
\end{equation}
$$[\MC] \mapsto \sum_{i,p} (-1)^{i} [\HC^i ( gr^F_{-p} DR(\MC) )] \cdot
(-y)^p\:.$$
We let ${td_{(1+y)}}_*$ be the natural transformation (cf. \cite{Y,BSY}):
\begin{equation}\label{td}
{td_{(1+y)}}_*:G_0(X) \otimes \Z[y, y^{-1}] \to H_{*}(X) \otimes
\Q[y, y^{-1}, (1+y)^{-1}]\;;
\end{equation}
$$[\mathcal{F}] \mapsto \sum_{k \geq
0} td_k([\mathcal{F}]) \cdot (1+y)^{-k}\:,$$
where $H_*(X)$  is either the Borel-Moore homology in even degrees $H^{BM}_{2*}(X)$, or the Chow group $CH_*(X)$, and $td_k$ is the
degree $k$ component of the Todd class
transformation $td_*:G_0(X) \to H_*(X) \otimes \Q$ of
Baum-Fulton-MacPherson \cite{BFM, Fu}, which is linearly extended over
$\Z[y, y^{-1}]$.

\begin{defn}\label{D1} The (motivic) Hirzebruch class
transformation ${MHT_y}_*$ is defined by the composition (cf.
\cite{BSY, Sch3})
\begin{equation}\label{IT0}
{MHT_y}_* :={td_{(1+y)}}_* \circ MHC_y: K_0(\mh(X)) \to H_{*}(X)
\otimes \Q[y,y^{-1},(1+y)^{-1}]\:.
\end{equation}
By a recent result of \cite{Sch3}[Prop.5.21], ${MHT_{y}}_*$ takes values in
$$H_{*}(X)\otimes \Q[y,y^{-1}]\subset H_{*}(X)
\otimes \Q[y,y^{-1},(1+y)^{-1}],$$ so that we consider it as a transformation
\begin{equation}\label{IT}
{MHT_y}_* :={td_{(1+y)}}_* \circ MHC_y: K_0(\mh(X)) \to H_{*}(X)
\otimes \Q[y,y^{-1}]\:.
\end{equation}
The (motivic) Hirzebruch class ${T_y}_*(X)$ of a complex algebraic
variety $X$ is then defined by \begin{equation}
{T_y}_*(X):={MHT_y}_*([\Q_X^H]).\end{equation} If $X$ is an $n$-dimensional complex algebraic manifold and $\LL$ is a local system on $X$ underlying an admissible variation of mixed Hodge structures (with quasi-unipotent monodromy at infinity), we define twisted characteristic classes by \begin{equation}
{T_y}_*(X;\LL):={MHT_y}_*([\LL^H]),
\end{equation} where $\LL^H[n]$ is the smooth mixed Hodge module on $X$ with underlying perverse sheaf $\LL[n]$.
Similarly, for $X$ pure-dimensional, we let  \begin{equation}{IT_y}_*(X):={MHT_y}_*([IC'^H_X])\end{equation} be the value of the transformation ${MHT_y}_*$ on the shifted intersection cohomology module $IC'^H_X:=IC^H_X[-{\rm dim}(X)]$. And if $\LL$ is an admissible variation defined on a smooth Zariski open and dense subset of $X$, we set \begin{equation}{IT_y}_*(X;\LL):={MHT_y}_*([IC'^H_X(\LL)]).\end{equation}
\end{defn}
\begin{remark}\rm Over a point, the transformation ${MHT_y}_*$ coincides with the
$\chi_y$-genus ring homomorphism $\chi_y:K_0(\mhs^p) \to
\Z[y,y^{-1}]$ defined on the Grothendieck group of (graded) polarizable mixed Hodge structures by 
\begin{equation} \chi_y([H]):=\sum_p {\rm dim} Gr^p_F(H \otimes \C) \cdot (-y)^p,
\end{equation}
for $F^{\centerdot}$ the Hodge filtration of $H \in \mhs^p$. Here we use the fact proved by Saito that there is an equivalence of categories $\mh(pt)\simeq\mhs^p$.
\end{remark}

By definition, the transformations $MHC_y$ and ${MHT_y}_*$   commute with proper
push-forward, and the following {\it normalization} property holds (cf.
\cite{BSY}): If $X$ is smooth and pure dimensional, then
\begin{equation}\label{nor} {T_y}_*(X)=T_y^*(TX) \cap [X]\:,\end{equation}
where $T_y^*(TX)$ is the
cohomology Hirzebruch class of $X$ (\cite{H}) defined via the power series  \begin{equation}
Q_y(\alpha)=\frac{\alpha(1+y)}{1-e^{-\alpha(1+y)}}-\alpha y \in
\Q[y][[\alpha]],
\end{equation}
that is, \begin{equation}\label{Hsm}
T_y^*(TX)=\prod_{i=1}^{dim(X)}Q_y(\alpha_i),
\end{equation}
where $\{\alpha_i\}$ are the Chern roots of the tangent bundle
$TX$. Note that for the values $y=-1$, $0$, $1$ of the parameter, the class $T_y^*$ reduces to the total Chern class $c^*$, Todd class $td^*$, and $L$-polynomial $L^*$, respectively.\\

It was shown in \cite{BSY} that for any variety $X$ the limits of ${T_y}_*(X)$
for $y=-1,0$ exist, with
\begin{equation}\label{-1}{T_{-1}}_*(X)=c_*(X) \otimes \Q\end{equation}
the total (rational) Chern
class of MacPherson (\cite{MP}). This was recently improved in \cite{Sch3}[Prop.5.21]
to the following result stating that for a complex of mixed Hodge modules $\MC\in D^b\mh(X)$,
\begin{equation}
{MHT_{-1}}_*([\MC]) = c_*([rat(\MC)]) =: c_*(\chi_{stalk}([rat(\MC)]))
\in H_*(X)\otimes\Q
\end{equation}
is the rationalized MacPherson Chern class of the underlying constructible sheaf complex
$rat(\MC)\in D^b_c(X)$ of $\MC$ (i.e., of the constructible function $\chi_{stalk}([rat(\MC)])\in F(X)$).\\

For a variety $X$ with at most ``Du Bois singularities"
(e.g., toric varieties), we have that
\begin{equation}\label{0}{T_0}_*(X)=td_*(X):=td_*([\mathcal{O}_X]) \:,\end{equation}
for  $td_*$ the
Baum-Fulton-MacPherson transformation \cite{BFM, Fu}. And it is still only conjectured that if $X$ is a compact algebraic variety, then ${IT_1}_*(X)$ is the Goresky-MacPherson $L$-class (cf. \cite{BSY}[Rem.5.4]). This is only known in some special cases, e.g., if $X$ has a small resolution of singularities.
If $X$ is projective, the degrees of these classes coincide by Saito's Hodge index theorem for intersection cohomology (see \cite{Sa0}[Thm.5.3.2]), i.e., the following identification holds
\begin{equation}\label{HIH} I\chi_1(X)=\sigma(X),
\end{equation}
for $\sigma(X)$ the Goresky-MacPherson signature of the projective variety $X$.
Also note that if $X$ is a {\it rational homology manifold} then $IC'^H_X\simeq \Q^H_X$, so that
in this case we get that ${IT_{y}}_*(X)={T_{y}}_*(X)$. As a byproduct of results obtained in this paper, we are able to prove the above conjecture for the case of a compact complex algebraic variety $X$ with only isolated singularities (or more generally, with a suitable singular locus, which is smooth with simply-connected components), which is a rational homology manifold that can be realized as a global hypersurface in a complex algebraic manifold; see Section \ref{fin}.

\section{Milnor-Hirzebruch classes of complex hypersurfaces}

\subsection{Milnor-Hirzebruch classes via specialization}
Let, as before, $X=\{f=0\}$ be an algebraic variety defined as the zero-set of codimension one of an algebraic function $f:M \to \C$, for $M$ a  complex algebraic manifold of complex dimension $n+1$.  Let $i:X \hookrightarrow M$ be the inclusion map. Denote by $L$ the trivial line bundle on $M$, obtained by pulling back by $f$ the  tangent bundle of $\C$. Then the virtual tangent bundle of $X$ can be identified with 
\begin{equation}\label{vir} T_{\rm vir}X=[TM\vert_X - L\vert_X]\:,\end{equation}
since $N_XM\simeq f^*N_{\{0\}}\C\simeq L\vert_X$.

Let $$\Psi_f^H, \Phi_f^H: {\rm MHM}(M) \to {\rm MHM}(X)$$ be the nearby and resp. vanishing cycle functors associated to $f$, which are defined on the level of Saito's algebraic mixed Hodge modules \cite{Sa0,Sa1}. These functors induce transformations on the corresponding Grothendieck groups and, by construction,  the following identity holds in $K_0({\rm MHM}(X))$ for any $[\MC]\in K_0(\mh(M))$:
\begin{equation}\label{triangle}
\Psi^H_f([\MC])= \Phi^H_f([\MC])-i^*([\MC])\: .
\end{equation}
Recall that, if $$rat: {\rm MHM}(X) \to {\rm Perv}_{\Q}(X)$$ is the forgetful functor assigning to a mixed Hodge module the underlying perverse sheaf, then $rat \circ \Psi^H_f={{^p}\Psi_f} \circ rat$ and similarly for $\Phi_f^H$. Here ${{^p}\Psi_f}:=\Psi_f[-1]$ is a shift of Deligne's nearby cycle functor \cite{De2}, and similarly for ${{^p}\Phi_f}$. 
So the shifted transformations $\Psi'^H_f:=\Psi^H_f[1]$ and $\Phi'^H_f:=\Phi^H_f[1]$
correspond under $rat$ to the usual nearby and vanishing cycle functors as stated in the Introduction in the commutative diagram (\ref{all-nearby}).\\

Let $i^!:H_*(M) \to H_{*-1}(X)$ denote the Gysin map between the corresponding  homology theories (see \cite{Fu,V}). The following is an easy consequence of the 
{\it specialization property} (\ref{sp}) of Sch\"{u}rmann \cite {Sc} for the Hirzebruch class transformation ${MHT_y}_*$ (cf. 
Thm.\ref{main-sch}):

\begin{lemma}\label{L1}
\begin{equation}\label{eq10} {T_y}^{\rm vir}_*(X):=T_y^*(T_{{\rm vir}}X) \cap [X]={MHT_y}_*(\Psi'^H_f(\left[\Q^H_M\right])).
\end{equation}
\end{lemma}
\begin{proof} Since $M$ is smooth, it follows that $\Q^H_M[n+1]$ is a mixed Hodge module, i.e. a complex concentrated in degree $0$. And since all our arguments are in Grothendieck groups, in order to simplify the notations, we will work with the shifted object $\Q^H_M \in \mh(M)[-n-1] \subset D^b\mh(M)$, whose class in $K_0(\mh(M))$ is identified with $$[\Q^H_M]=(-1)^{n+1}\cdot \left[\Q^H_M[n+1]\right].$$ 
By applying the identity (\ref{sp}) to the class $[\Q_M^H]\in K_0(\mh(M))$ we have that 
$${MHT_y}_*(\Psi'^H_f(\left[\Q^H_M\right]))=i^! {MHT_y}_*([\Q^H_M])=i^!{T_y}_*(M)=i^!(T_y^*(TM) \cap [M])\:,$$
where the last identity follows from the normalization property (\ref{nor}) of (motivic) Hirzebruch classes as $M$ is smooth. Moreover, by the definition of the Gysin map, the last term of the above identity becomes $i^*(T_y^*(TM)) \cap [X]$, which by the identification in (\ref{vir}) is simply equal to ${T_y}^{\rm vir}_*(X)$.

\end{proof}

We can now prove the following key result on the characterization of the Milnor-Hirzebruch class $\MC{T_y}_*(X)$:

\begin{thm}\label{M} The Milnor-Hirzebruch class of a globally defined hypersurface $X=f^{-1}(0)$ (of codimension one) in a complex algebraic manifold $M$ is entirely determined by the vanishing cycles of $f:M \to \C$. More precisely,  
\begin{equation}\label{eq20}
\MC{T_y}_*(X):={T_y}^{\rm vir}_*(X) - {T_y}_*(X)={MHT_y}_*(\Phi'^H_f(\left[\Q^H_M\right])).
\end{equation}
\end{thm}
\begin{proof} 
By applying the identity (\ref{triangle}) to the class $\left[\Q^H_M\right]$ of the constant Hodge sheaf  on $M$, we obtain the following equality in $K_0({\rm MHM}(X))$:
\begin{equation}\label{eq30}
\Phi_f^H(\left[\Q^H_M\right])]= \Psi_f^H(\left[\Q^H_M\right])+[\Q^H_X].
\end{equation}
The desired identity follows now from Lemma \ref{L1} after applying the natural transformation ${MHT_y}_*$ to equation (\ref{eq30}).

\end{proof}

Since the rational complex $\Phi^H_f(\Q_M)$ is supported only on the singular locus $X_{\rm sing}$ of $X$ (i.e., on the set of points in $X$ where the differential $df$ vanishes), the result of Theorem \ref{M} shows that the difference ${T_y}^{\rm vir}_*(X) - {T_y}_*(X)$ can be expressed entirely only in terms of invariants of the singularities of $X$. Namely, by the functoriality of the transformations ${T_y}_*$ and ${MHT_y}_*$ (for the closed inclusion $X_{\rm sing}\hookrightarrow X$), we can view
\begin{equation}\label{loc}\MC{T_y}_*(X):= {T_y}_*(\Phi^m_f([id_M]))=
{MHT_y}_*(\Phi'^H_f(\left[\Q^H_M\right])) \in H_*(X_{\rm sing})\otimes\Q[y]\end{equation}
as a {\it localized} Milnor-Hirzebruch class.
Therefore, we have the following
\begin{cor}
The classes ${T_y}^{\rm vir}_*(X)$ and ${T_y}_*(X)$ coincide in dimensions greater than the dimension of the singular locus of $X$.
\end{cor}

\begin{remark}\rm The nearby and vanishing cycle functors $\Psi_f^H, \Phi_f^H: {\rm MHM}(M) \to {\rm MHM}(X)$ have a functor automorphism $T_s$ of finite order, induced by the semisimple part of the monodromy $T$.
We have the decomposition $\Psi_f^H=\Psi^H_{f,1} \oplus \Psi^H_{f,\neq 1}$ such that $T_s=id$ on $\Psi^H_{f,1}$ and $1$ is not an eigenvalue of $T_s$ on $\Psi^H_{f,\neq 1}$, and similar for $\Phi_f^H$. By further decomposition into generalized eigenspaces the action of $T_s$ on the (complexification of the) mixed Hodge structures $\HC^j i_x^* \Phi^H_f \MC$ ($j \in \Z$), for $i_x:\{x\} \hookrightarrow X$ the inclusion of a point,  Saito \cite{Sa3} defined the {\it spectrum $hsp(\MC,f,x)$ of a (complex of) mixed Hodge module(s)} $\MC\in K_0(\mh(M))$, which is a generalization of Steenbrink's {\it Hodge spectrum} for hypersurface singularities \cite{St2,St4,Va} (see also \cite{GLM} for motivic analogues of vanishing cycles and Hodge spectrum). In this note we do not need to take into account these monodromy functors. However, the Hirzebruch-type invariants associated to the local Milnor fibers which appear in our formulae can be, in fact, computed from this well-studied Hodge spectrum information (see Remark \ref{spec} below).
\end{remark}

\begin{remark}\label{deg} (The degree of Milnor-Hirzebruch class) \rm\newline If $f:M \to \C$ is proper, the degree of the (zero-dimensional piece of the) Milnor-Hirzebruch class is computed by 
\begin{equation}\label{pd}
{\rm deg} \left( \MC {T_y}_*(X)\right) := \int_{[X]} {T_y}^{\rm vir}_*(X) - {T_y}_*(X)=\chi_y(X_t) - \chi_y(X),
\end{equation}
with $X_t:=f^{-1}(t)$ (for $t \neq 0$ small enough) the generic fiber of $f$.  In order to see this, first note that by pushing-down under $Rf_*$ the specialization identity $${MHT_y}_*(\Psi'^H_f([\Q^H_M]))=i^! {MHT_y}_*([\Q^H_M]),$$ one obtains the equality between the Hodge polynomial associated to the {\it limit mixed Hodge structure} on the cohomology of the {\it canonical fiber} $X_{\infty}$ (e.g., see \cite{PS}, \S 11), i.e.,  $$\chi_y(X_{\infty}):=-\chi_y([H^*(X;\Psi_f^H \Q^H_M)])$$ and respectively that of the nearby (smooth) fiber of $f$, $\chi_y(X_t)$. Then (\ref{pd}) follows by pushing-down under $Rf_*$ the identity of (\ref{triangle}) and then applying the transformation ${MHT_y}_*$ (which in this case reduces to the ring homomorphism $\chi_y$); compare with \cite{CLMS}[Sect.3.2].
Therefore, the formulae obtained in this note are indeed characteristic class generalizations of the results from \cite{CLMS}, as mentioned in the Introduction of the present paper.
\end{remark}

\subsection{Computational aspects. Examples}

We now illustrate by simple examples how one can explicitly compute the Milnor-Hirzebruch class $\MC{T_y}_*(X)$ in terms of invariants of the singular locus.

\begin{example}\label{ex:iso} \ {\rm Isolated singularities.} \newline
If the hypersurface $X$ has only isolated singularities, the corresponding vanishing cycles complex $\phi_f\Q_M$ is supported only at these singular points, and by Theorem \ref{M} we obtain:
\begin{equation}\label{iso} 
\MC{T_y}_*(X)=\sum_{x \in X_{\rm sing}} \chi_y(i_x^*\Phi'^H_f([\Q^H_M]))=\sum_{x \in X_{\rm sing}} (-1)^{n} \chi_y([\tilde{H}^n(F_x;\Q)]),
\end{equation}
where $i_x:\{x\} \hookrightarrow X$ is the inclusion of a point, and $F_x$ is the Milnor fiber of the isolated hypersurface singularity $(X,x)$ (which in this case is $(n-1)$-connected).
\end{example}

\begin{remark}\label{spec} (Hodge polynomials vs. Hodge spectrum) \rm\newline
Let us now point out the precise relationship between the Hodge spectrum and the less-studied Hodge polynomial of the Milnor fiber of a hypersurface singularity. We follow here notations and sign conventions similar to those from \cite{GLM}. Denote by $\mhs^{\rm mon}$ the abelian category of mixed Hodge structures endowed with an automorphism of finite order, and by $K^{mon}_0(\mhs)$ the corresponding Grothendieck ring. There is a natural linear map called the {\it Hodge spectrum}, $${\rm hsp}: K^{\rm mon}_0(\mhs) \to \Z[\Q] \simeq \bigcup_{n \geq 1} \Z[t^{1/n},t^{-1/n}],$$ such that
\begin{equation}\label{hsp}
{\rm hsp}([H]):=\sum_{ \alpha \in \Q \cap [0,1)} t^{\alpha} \left( \sum_{p \in \Z} {\rm dim}(Gr^{p}_F H_{\C,\alpha}) t^p \right)
\end{equation}
for any mixed Hodge structure $H$ with an automorphism $T$ of finite order, where $H_{\C}$ is the underlying complex vector space of $H$, $H_{\C,\alpha}$ is the eigenspace of $T$ with eigenvalue 
$\exp(2\pi i \alpha)$, and $F$ is the Hodge filtration on $H_{\C}$.
It is now easy to see that the $\chi_y$-polynomial of $H$ is obtained from ${\rm hsp}([H])$ by equating to $1$ the parameter $t$ corresponding to fractional powers $\alpha \in \Q \cap [0,1)$, and by setting the $t$ of integer powers be equal to $-y$.  Lastly, the Hodge spectrum of hypersurface singularities (where one applies the above construction for the cohomology of the Milnor fiber endowed with the action of the semisimple part of the monodromy) has been studied in many cases, e.g., for isolated weighted homogeneous hypersurface singularities (\cite{St3}) or isolated hypersurface singularities with non-degenerate Newton polyhedra (\cite{St2,Sa4}), but see also \cite{Ku,NS}. (For the relation to the original definition of Steenbrink of the Hodge spectrum see e.g. \cite{Ku}[Sect.8.10].) In all these cases, we can therefore compute the $\chi_y$-polynomials appearing in our formulae. (In fact for isolated
hypersurface singularities the corresponding spectrum can also be calculated by computer programs, e.g. see \cite{Slz}).
\end{remark}

\begin{example}\label{smooth} \ {\rm Smooth singular locus.} \newline
At the other extreme, let us now assume that $X$ has a smooth singular locus $\Sigma$, which for simplicity is assumed to be connected. Moreover, suppose that $\phi_f\Q_M$ is a constructible complex with respect to the stratification of $X$ given by the strata $\Sigma$ and $X \setminus \Sigma$ (e.g., this is the case if the filtration $\Sigma \subset X$ corresponds to a Whitney stratification of $X$). If $r={\rm dim}_{\C} \Sigma < n$, the Milnor fiber $F_x$ at a point $x \in \Sigma$ has the homotopy type of an  $(n-r)$-dimensional CW complex, which moreover is $(n-r-1)$-connected. So in $K_0(\mh(X))$ the following identification holds: $$\Phi^H_f([\Q^H_M])=(-1)^{n-r-1} \cdot [\LL^H_{\Sigma}],$$ for $\LL_{\Sigma}$ the admissible variation of mixed Hodge structures  (on $\Sigma$) with stalk at $x \in \Sigma$ given by $H^{n-r}(F_x;\Q)$. Therefore, Theorem \ref{M} yields that:
\begin{equation}\label{sm}
\MC{T_y}_*(X)=(-1)^{n-r} \cdot {T_y}_*(\Sigma;\LL_{\Sigma}), 
\end{equation}
with ${T_y}_*(\Sigma;\LL_{\Sigma}):={MHT_y}_*([\LL^H_{\Sigma}])$ the twisted characteristic class corresponding to the admissible variation $\LL_{\Sigma}$ on $\Sigma$ (cf. Def. \ref{D1}). Formulae describing the calculation of such classes are obtained in the authors' papers \cite{CLMS2,CLMS,MS,Sch3}. In particular, if $\pi_1(\Sigma)=0$, formula (\ref{sm}) reduces to:
\begin{equation}\label{smm}
\MC{T_y}_*(X)=(-1)^{n-r} \cdot \chi_y([H^{n-r}(F_x;\Q)]) \cdot {T_y}_*(\Sigma), 
\end{equation}
which is just a particular case of formula (\ref{eq1}). 

Note that, if $N$ is a normal slice to $\Sigma$ at $x$ (i.e., $N$ is a smooth analytic subvariety  of $M$, transversal to $\Sigma$ at $x$), it follows that \begin{equation}\label{slice} \chi_y([H^{n-r}(F_x;\Q)])=\chi_y([H^{n-r}(F_{N,x};\Q)]),\end{equation} where $F_{N,x}$ is the Milnor fiber of the isolated singularity germ $(X \cap N,x)$ defined (locally in the analytic topology) by restricting $f$ to a normal slice $N$ at $x$. Indeed, by \cite{DMST}[Cor.1.5], the spectrum, thus the $\chi_y$-polynomial, is preserved by restriction to a normal slice. (Here, our sign conventions in the definition of the spectrum cancel out the sign issues appearing in \cite{DMST}.) In particular, this ``normal" information to the singular stratum is computable as mentioned in Remark \ref{spec}.
\end{example}

Before giving a very concrete example, we begin with the following considerations. 
Let $f: \C^{n+1} \to \C$ be a polynomial function, and denote the coordinates of $\C^{n+1}$ by $x_1,\cdots, x_{n+1}$. Assume $f$ depends only on the first $n-k+1$ coordinates $x_1,\cdots,x_{n-k+1}$, and 
it has an isolated singularity at $0\in \C^{n-k+1}$ when regarded as a polynomial function on $ \C^{n-k+1}$. If $X:=f^{-1}(0) \subset \C^{n+1}$, then the singular locus $\Sigma$ of $X$ (or $f$) is the affine space $\C^k$ corresponding to the remaining coordinates $x_{n-k+2},\cdots,x_{n+1}$ of $\C^{n+1}$, and the filtration $\Sigma \subset X$ induces a Whitney stratification of $X$. The transversal singularity in the normal direction to $\Sigma$ at a point $x \in \Sigma$ is exactly the isolated singularity at $0 \in \C^{n-k+1}$ mentioned above. 
Since $\Sigma$ is smooth and simply-connected, we get by Example \ref{smooth} the identity
$${\MC T_y}_*(X)=(-1)^{n-k}{\chi_y([\tilde{H}^{n-k}(F_0;\Q)])} \cdot [\C^k] \in H_*(X) \otimes \Q[y],$$
with $F_0$ the Milnor fiber of $f: \C^{n-k+1}\to \C$ at $0$.

Let us now assume that the above isolated singularity in $\C^{n-k+1}$ is a Brieskorn-Pham singularity, i.e., defined by 
$$f(x_1,...,x_{n-k+1}):= \sum_{j=1}^{n-k+1} x_j^{w_j}$$
with $w_j \geq 2$. Then by the Thom-Sebastiani theorem (e.g., see \cite{GLM}[Thm.5.18]) one has the following computation of the Hodge spectrum:
\begin{equation}{\rm hsp}([\tilde{H}^{n-k}(F_0;\Q)])=
\prod_{j=1}^{n-k+1} (\prod_{i=1}^{w_j-1} t^{i/w_j}).\end{equation}
And this can be specialized by Remark \ref{spec} to a calculation of the $\chi_y$-polynomial of $F_0$.  
 
In particular, by applying the above considerations to the polynomial $f:\C^{n+1} \to \C$ given by $$f(x_1,\cdots,x_{n+1})=(x_1)^2+\cdots (x_{n-k+1})^2, \ \ \ k \geq 0,$$ 
we obtain for $X:=f^{-1}(0)$ that $${\MC}T_{y*}(X)=(-y)^{\lceil \frac{n-k}{2} \rceil} \cdot [\C^k] \in H_*(X) \otimes \Q[y],$$ where ${\lceil - \rceil}$ denotes the rounding-up to the nearest integer.

\begin{example}\label{topdeg} \ {\rm The top degree of the Milnor-Hirzebruch class.} \newline
Let $\Sigma:=X_{\rm sing}$ be the singular locus of $X$, and denote by $\Sigma_{\rm reg}:=\Sigma\backslash (X_{\rm sing})_{\rm sing}$ its regular part. Assume for simplicity that $\Sigma$ is irreducible. Then, if $r:={\rm dim}_{\C} \Sigma$,  the long exact sequence in Borel-Moore homology
$$\cdots \to H_*^{BM}(\Sigma_{\rm sing}) \to H_*^{BM}(\Sigma) \to H_*^{BM}(\Sigma_{\rm reg}) \to H_{*-1}^{BM}(\Sigma_{\rm sing}) \to \cdots$$ yields the isomorphism \begin{equation} \label{to}
H_{2r}^{BM}(\Sigma) \simeq H_{2r}^{BM}(\Sigma_{\rm reg}).
\end{equation}
And since $\Sigma_{\rm reg}$ is smooth and connected, we get by Poincar\'e Duality that $$H_{2r}^{BM}(\Sigma_{\rm reg}) \simeq H^0(\Sigma_{\rm reg}) \simeq \Z,$$ and also 
$$H_{i}^{BM}(\Sigma) \simeq H_{i}^{BM}(\Sigma_{\rm reg}) \simeq 0, \ {\rm for} \ i>2r.$$ Therefore, 
$H^{BM}_{top}(\Sigma)\simeq \Z$, and is generated by the fundamental class $[\Sigma]$. 

The top degree of the Milnor-Hirzebruch class lies in $H_{top}(\Sigma) \otimes \Q[y]$, where $H_{top}(\Sigma)$ denotes as before either the top Borel-Homology group or the top Chow group. In fact, note that there is a group isomorphism $CH_r(\Sigma) \overset{\simeq}{\to} H_{2r}^{BM}(\Sigma)$. So, we can write:
\begin{equation}
\MC{T_y}_*(X)=m_{\Sigma}(y) \cdot [\Sigma] +  {\rm ``lower \ terms"} \  \in H_{top}(\Sigma) \otimes \Q[y] \oplus \cdots ,
\end{equation}
where $m_{\Sigma}(y)$ denotes the multiplicity of the Milnor-Hirzebruch class along (the regular part of) $\Sigma$. This multiplicity can be computed (locally, in the analytic topology) in a normal slice $N$ at a point $x \in \Sigma_{\rm reg}$. And just as in Example \ref{smooth}, it follows that 
\begin{equation}
m_{\Sigma}(y) =(-1)^{n-r}\cdot \chi_y([H^{n-r}(F_{N,x};\Q)]),\end{equation} 
where $F_{N,x}$ is the Milnor fiber of the isolated singularity germ $(X \cap N,x)$ defined (locally in the analytic topology) by restricting $f$ to a normal slice $N$ at $x \in \Sigma_{\rm reg}$.
\end{example}

\begin{remark}\rm
In general, 
for $\Sigma$ an $r$-dimensional irreducible component of $X_{\rm sing}$, one has
canonical arrows (factorising the isomorphism (\ref{to}) above)
$$H^{BM}_{2r}(\Sigma) \to H^{BM}_{2r}(X_{\rm sing})
\to H^{BM}_{2r}(\Sigma_{\rm reg}),$$
so that the first arrow is injective. Therefore the arguments of Example \ref{topdeg}
can be applied to all irreducible components of the singular locus of $X$.
Specializing further to $y=-1$, we get that the corresponding
``top-dimensional" multiplicity
of the localized Milnor class along $\Sigma$ is given by the Euler characteristic
$\chi(\tilde{H}^*(F_{N,x};\Q))$ of the Milnor fiber in a transversal slice.
This fits with the
corresponding result of \cite{BLSS2}, but it was not explicitly stated  in \cite{Sch,Sch2}.
\end{remark}

In more general situations, the calculation of the Milnor-Hirzebruch class of $X$ requires a better understanding of the  delicate monodromy problem, as we shall explain below. First note that we can describe the Grothendieck group $K_0(\mh(X))$ of mixed Hodge modules on $X$ as: \begin{equation}\label{K0} K_0(\mh(X))= K_0(\hp(X)^p), \end{equation} where $\hp(X)^p$ denotes the abelian category of pure polarizable Hodge modules \cite{Sa0}. And by the decomposition by strict support of pure Hodge modules, it follows that $K_0(\hp(X)^p)$ is generated by elements of the form $[IC_S^H(\LL)]$, for $S$ an irreducible closed subvariety of $X$ and $\LL$ a polarizable variation of Hodge structures (admissible at infinity) defined on a smooth Zariski open and dense subset of $S$. Thus the image of the natural transformation ${MHT_y}_*$ is generated by { twisted} characteristic classes $${IT_y}_*(S;\LL):={MHT_y}_*([IC'^H_S(\LL)]),$$ with $IC'^H_S(\LL):=IC_S^H(\LL)[-{\rm dim}_{\C}(S)]$, and $S$ and $\LL$ as above. Moreover, since by Theorem~\ref{M}, the Milnor-Hirzebruch class is supported only on the singular locus 
$X_{\rm sing}$ of $X$, the class $\MC{T_y}_*(X)$ is generated only by classes of the form ${IT_y}_*(S;\LL)$ with $S$ an irreducible closed subvariety contained in $X_{\rm sing}$, and with $\LL$ as above.  The calculation of such twisted characteristic classes is in general very difficult. Results in this direction, usually referred to as ``Atiyah-Meyer type formulae", are described in some special cases in \cite{CLMS2,CLMS,MS,Sch3}.

Another set of generators for the Grothendieck group $K_0(\mh(X))$ can be obtained by using resolutions of singularities. More precisely, $K_0(\mh(X))$ is generated by elements of the form $[p_*(j_*\LL')]$ (or $[p_*(j_!\LL')]$) , with $p:Z \to X$ a proper algebraic map from a smooth algebraic manifold $Z$, $j:U=Z\setminus D \hookrightarrow Z$ the open inclusion of the complement of a normal crossing divisor $D$ with smooth irreducible components, and $\LL'$ an admissible variation of mixed Hodge structures on $U$. By the functoriality of ${MHT_y}_*$, it suffices to understand the characteristic classes of the form ${MHT_y}_*(j_*\LL')$ (or ${MHT_y}_*(j_!\LL')$), with $j$ and $\LL'$ as above. Such classes can be computed in terms of the twisted logarithmic de Rham complex associated to the Deligne extension of $\LL'$ to $(Z,D)$. For generators of the form $[p_*(j_*\LL')]$, the corresponding classes are calculated in \cite{CLMS2,CLMS,MS,Sch3}. Similar arguments apply to the calculation of classes associated to generators $[p_*(j_!\LL')]$, but using  a different Deligne extension, with residues in the half-open interval $(0,1]$; compare  \cite{Sa1}[Sect.3.10-3.11].

We conclude this section with a discussion on the following situation.
\begin{example} \ {\rm One-dimensional singular locus.} \newline
Assume the singular locus $X_{\rm sing}$ of the hypersurface $X$ is one-dimensional, and consider an adapted stratification
$S\subset X_{\rm sing}$ with $S$ zero-dimensional. If $i:S \hookrightarrow X_{\rm sing}$ denotes 
the inclusion map, and $j$ is the inclusion of the open complement of $S$ in $X_{\rm sing}$, then 
by using the distinguished triangle
$j_!j^* \to id \to i_*i^* \to $  applied to
$\Phi'^H_f([\Q^H_M])$, one can reduce the calculation of 
${\MC T_y}_*(X)$ to the following:
\begin{enumerate}
\item the calculation of $\chi_y([\tilde{H}^*(F_x;\Q)]$ at the isolated points
$x\in S$. These points are in general non-isolated singularities of $X$,
but their spectrum can be reduced to the calculation of spectrum of isolated hypersurface
singularities defined by deformations $f+g^N$, for $g$ a generic linear form. This is the content of Steenbrink's conjecture \cite{St4}, proved in the general case by Saito \cite{Sa3} (cf. also \cite{GLM}[Thm.6.10]).
\item classes of the form ${MHT_y}_*(j_!\LL')$ wich can be calculated as sketched above.
\end{enumerate}
\end{example}

\subsection{Computation of Milnor-Hirzebruch classes by Grothendieck calculus}\label{sec.comp}
We now turn to the proof of the second part of Theorem \ref{main} from the Introduction, where for simplicity we assume that the monodromy contributions along all strata in a stratification of $X$ are trivial, e.g., all strata are simply-connected. This assumption allows us to identify the coefficients in the above generating sets of $K_0(\mh(X))$, and to obtain precise formulae for the Milnor-Hirzebruch class as a direct application of the specialization property (\ref{sp}) combined with standard calculus in Grothendieck groups. We begin by recalling some useful results from \cite{CMS2}. 
Let $X$ be a pure-dimensional complex algebraic variety endowed with a
complex algebraic Whitney stratification $\VV$ so that the
intersection cohomology complexes $$IC'_{\bar W}:=IC_{\bar W}[-{\rm
dim}(W)]$$ are $\VV$-constructible for all strata $W \in \VV$. Let us fix for each
$W \in \VV$ a point $w \in W$ with inclusion $i_w:\{w\}
\hookrightarrow X$. Then
\begin{equation}\label{100} i_w^*[IC'^H_{\bar W}]=[i_w^*IC'^H_{\bar
W}]=[\Q^H_{pt}]\in K_0(MHM(w))=K_0(MHM(pt)),\end{equation} and
$i_w^*[IC'^H_{\bar V}] \neq [0] \in K_0(MHM(pt))$ only if $W \subset 
{\bar V}$. Moreover,  in this case we have that for any $j \in \Z$, 
\begin{equation}\label{cone} \HC^j (i_w^*IC'_{\bar V}) \simeq
IH^j(c^{\circ}L_{W,V}),\end{equation} for $c^{\circ}L_{W,V}$ the
open cone on the {\it link} $L_{W,V}$ of $W$ in $\bar V$.  So
$$i_w^*[IC'^H_{\bar V}]=[IH^*(c^{\circ}L_{W,V})] \in K_0(MHM(pt)),$$
with the mixed Hodge structures on the right hand side defined by
the isomorphism (\ref{cone}). For future reference, let us set:
$$I\chi_y(c^{\circ}L_{W,V}):=\chi_y([IH^*(c^{\circ}L_{W,V})]).$$
One of the main results of \cite{CMS2} can now be stated as follows:
\begin{thm}\label{CMS} \ (\cite{CMS2}, Theorem 3.2)\newline Let $\VV_0$ be the set of all singular strata of $X$, i.e., strata $V\in \VV$ so that ${\rm dim}(V)<{\rm dim}(X)$. 
For each $V \in \VV_0$ define inductively
\begin{equation}\label{eq8}
\widehat{IC^H}(\bar V):=[IC'^H_{\bar V}] - \sum_{W < V}
\widehat{IC^H}(\bar W) \cdot i_w^* [IC'^H_{\bar V}] \in
K_0(MHM(X)), 
\end{equation} where the summation is over all strata $W \subset {\bar V} \setminus V$. 
Assume $[\MC] \in K_0(MHM(X))$ is an
element of the $K_0(MHM(pt))$-submodule $\langle [IC'^H_{\bar V}]
\rangle$ of $K_0(MHM(X))$ generated by the elements $[IC'^H_{\bar
V}]$, $V \in \VV$. Then we have the following equality  in
$K_0(MHM(X))$:
\begin{equation}\label{mE}
[\MC]=\sum_{S\in \pi_0(X_{\rm reg})} [IC'^H_{\bar S}] \cdot i_s^*[\MC]+\sum_{V \in \VV_0}  \widehat{IC^H}(\bar V)
\cdot \left( i_v^*[\MC] -\sum_{S\in \pi_0(X_{\rm reg})} i_s^*[\MC] \cdot i_v^*[IC'^H_{\bar S}] \right),
\end{equation}
where $\pi_0(X_{\rm reg})$ stands for the set of connected components of the regular (top dimensional) stratum in $X$.
\end{thm}

Before indicating how Theorem \ref{CMS} can be employed for proving our result, let us remark that one instance when the technical hypothesis $[\MC] \in \langle [IC'^H_{\bar V}]
\rangle$ is satisfied for a fixed $\MC \in D^b{\rm MHM}(X)$, is when all strata $V \in \VV$ are simply-connected and the rational complex $rat(\MC)$ is $\VV$-constructible. For this fact, we refer to \cite{CMS2}[Ex.3.3] where more general situations are also considered. Also note that Theorem \ref{CMS} above is stated in a slightly more general form than the corresponding result of \cite{CMS2}, where only the case of an irreducible variety $X$ was needed. However, the proof is identical to that of Theorem 3.2 of \cite{CMS2}, so we omit it here.

\begin{remark}\rm  Note that if under the hypotheses of Theorem \ref{CMS}, we assume moreover that $\MC \in D^b{\rm MHM}(X)$ is in fact supported only on the collection of singular strata $\VV_0$, then equation (\ref{mE}) reduces to
\begin{equation}\label{mEs}
[\MC]= \sum_{V \in \VV_0}  \widehat{IC^H}(\bar V) \cdot  i_v^*[\MC].
\end{equation}
\end{remark}

We can now prove the second part our main theorem, which we recall here for the convenience of the reader. 
\begin{thm}\label{main2} Let $X=\{f=0\}$ be a complex algebraic variety defined as the zero-set (of codimension one)  of an algebraic function $f:M \to \C$, for $M$ a  complex algebraic manifold. Fix a Whitney stratification $\VV$ on $X$, and denote by $\VV_0$ the collection of all singular strata (i.e., strata $V \in \VV$ with ${\rm dim}(V)<{\rm dim}(X)$). For each $V \in \VV_0$, define inductively $$\widehat{IT}_y(\bar V):= {IT_y}_*(\bar V)- \sum_{W < V}
\widehat{IT}_y(\bar W) \cdot I\chi_y(c^{\circ}L_{W,V}),$$ where the summation is over all strata $W \subset {\bar V} \setminus V$ and
$c^{\circ}L_{W,V}$ denotes the open cone on the link of  $W$ in
$\bar{V}$. (As the notation suggests, the class $\widehat{IT}_y(\bar V)$ depends only on the complex algebraic variety ${\bar V}$ with its induced algebraic Whitney stratification.) Then, if all strata $V \in \VV_0$ are assumed to be simply-connected, the following holds:
\begin{equation}\label{eq40}
\MC{T_y}_*(X):={T_y}^{\rm vir}_*(X) - {T_y}_*(X)=\sum_{V \in \VV_0} \widehat{IT}_y(\bar V) \cdot \chi_y([\tilde{H}^*(F_v;\Q)]),
\end{equation} for $F_v$ the Milnor fiber of a point $v \in V$.
\end{thm}
\begin{proof} By using the equation (\ref{eq20}), it suffices to show that:
\begin{equation}\label{eq50} {MHT_y}_*(\Phi'^H_f(\left[\Q^H_M\right]))=\sum_{V \in \VV_0} \widehat{IT}_y(\bar V) \cdot \chi_y([\tilde{H}^*(F_v;\Q)])\end{equation}
Next note that the sheaf complex $\Phi_f(\Q_M)$ is supported only on singular strata of $X$ and,  moreover, if $v \in V \in \VV_0$ then the following identity holds in $K_0({\rm MHM}(pt))$:
\begin{equation}\label{eq60}i_v^*\Phi'^H_f(\left[\Q^H_M\right])=
[\HC^*(\Phi'^H_f(\Q^H_M))_v]=[\tilde{H}^{*}(F_v;\Q)],\end{equation} where $F_v$ is the Milnor fiber of $f$ at $v$.

By using the fact that the transformation ${MHT_y}_*$ commutes with the exterior product 
$$K_0({\rm MHM}(X)) \times K_0({\rm MHM}(pt)) \to K_0({\rm MHM}(X\times \{pt\}))\simeq K_0({\rm MHM}(X))$$
(see \cite{Sch3}[Sect.5]),
it is easy to see that for each $V \in \VV_0$ the characteristic class $\widehat{IT}_y(\bar V)$ is  just ${MHT_y}_*(\widehat{IC^H}(\bar V))$.
Then (\ref{eq50}) follows by applying ${MHT_y}_*$ to the identity (\ref{mEs}), together with the identification in (\ref{eq60}), and the fact that ${MHT_y}_*$ commutes with the exterior product.

\end{proof}

\begin{remark}\rm By using $[\MC]=[\Q^H_X]$ in the identity (\ref{mE}), and after applying the transformation ${MHT_y}_*$, we obtain the following relationship between the classes ${T_y}_*(X)$ and ${{IT}_y}_*(X)$, respectively:
\begin{equation}\label{eq70}
{T_y}_*(X) -{IT_y}_*(X)=\sum_{V \in \VV_0} \widehat{IT}_y(\bar V) \cdot (1-\chi_y([IH^*(c^{\circ}L_{V,X})])),
\end{equation} for $L_{V,X}$ the link of the stratum $V$ in $X$.  Here we use the fact that for a pure-dimensional algebraic variety $X$, $$IC'^H_X=\oplus_{S\in \pi_0(X_{\rm reg})} IC'^H_{\bar S},$$ thus by taking stalk cohomologies we get $$[IH^*(c^{\circ}L_{V,X})]=\oplus _{S\in \pi_0(X_{\rm reg})} [IH^*(c^{\circ}L_{V,S})] \in K_0(\mh(pt)).$$ The identity (\ref{eq70}) can be used to express the Milnor-Hirzebruch class ${T_y}^{\rm vir}_*(X) - {T_y}_*(X)$ as a weighted sum of classes ${T_y}_*({\bar V})$ of closures of singular strata, the weights depending entirely on the information encoded in the normal direction to the respective strata. More precisely, we obtain the first part of our main  Theorem \ref{main}:
\begin{equation}\label{eq90}
\MC{T_y}_*(X)=\sum_{V \in \VV_0}\left( {T_y}_*(\bar V) - {T_y}_*({\bar V} \setminus V) \right) \cdot \chi_y([\tilde{H}^*(F_v;\Q)]),
\end{equation} for $F_v$ the Milnor fiber of a point $v \in V$. Finally, we obtain the first part of Theorem \ref{main} in the stronger form indicated in Remark \ref{new} by using ``rigidity" and multiplicativity for exterior products with points. More precisely, it follows by ``rigidity" (e.g., see \cite{CMS2}[p.435]) that a ``good"  variation of mixed Hodge structures on a connected complex
algebraic manifold $V$ is a constant variation provided  the underlying
local system is already constant. Applying this fact to a ``good" variation $\LL_V$ with constant underlying local system on a connected stratum $V \in \VV_0$, we get that $\LL_V\simeq k^*\LL_v$, where $v\in V$ is a point in the stratum, and $k: V\to v$ is the constant map. Therefore, if $j: V\to \bar{V}$ denotes the open inclusion into the closure of the stratum, we have that
$$j_!\LL_V \simeq j_!k^*\LL_v \simeq j_!\Q^H_V \boxtimes \LL_v \:.$$
Then the claim of Remark \ref{new} follows from the multiplicativity of
${MHT_y}_*(-)$ with respect to exterior products (with points).

\end{remark}

\subsection{Intersection Milnor-Hirzebruch classes} By analogy with the Milnor-Hirzebruch class, we can define {\it intersection Milnor-Hirzebruch classes} for a (pure-dimensional) complex hypersurface as the difference
\begin{equation}\MC{IT_y}_*(X):={T_y}^{\rm vir}_*(X) - {IT_y}_*(X).\end{equation} 
In fact, this is more natural to consider if one wants to compare the specialization at $y=1$ of ${\MC{IT}_y}_*(X)$
with the difference term $L^{\rm vir}_*(X)-L_*(X)$ of the corresponding $L$-classes, since
$L_*(X)$ is defined with the help of the shifted (self-dual) intersection cohomology complex 
$IC'_X:=IC_X[-{\rm dim}(X)]$ of $X$.

A direct interpretation for this class can be given by noting that (compare \cite{Sa}[p.152-153]) $IC^H_X$ is a {\em direct summand} of
$Gr^W_{n}\Psi^H_{f}(\Q_M^H[n+1])\in \hp(X)$,
where $W$ is the weight filtration on $\Psi^H_{f}$. 
In fact, $\Q_M^H[n+1]\in \hp(M)$ is a pure Hodge module of weight $n+1$ (with strict support 
$M$), so that by the inductive definition of pure Hodge modules (\cite{Sa0, Sa1}) $$Gr^W_{n}\Psi^H_{f}(\Q_M^H[n+1])\in \hp(X)$$
is a pure Hodge module of weight $n$. So it is a finite direct sum of pure Hodge modules
of weight $n$ with strict support in irreducible subvarieties of $X$.
But $$\Psi^H_{f}(\Q_M^H[n+1])\vert_{X_{\rm reg}}\simeq \Q_{X_{\rm reg}}^H[n]\:,$$
therefore  $IC^H_X$ has to be the direct summand of $Gr^W_{n}\Psi^H_{f}(\Q_M^H[n+1])$
coming from the pure direct summands with strict support the irreducible components of $X$.
Then 
\begin{equation}\label{last} \begin{split}
(-1)^n \cdot {MHT_y}_*&\left([Gr^W_{n}\Psi^H_{f}(\Q_M^H[n+1])\ominus IC^H_X]+
 \sum_{k \neq n} [Gr^W_{k}\Psi^H_{f}(\Q_M^H[n+1])] \right)\\
 &=\MC{IT_y}_*(X) \in H_*(X)\otimes \Q[y]\subset H_*(X)\otimes \Q[y,y^{-1}] \: . \end{split}\end{equation} 
Here the last inclusion follows by $T_*(X)\in  H_*(X)\otimes \Q[y]$ from the commutative diagram (\ref{special}).

Formula (\ref{last}) holds independently of any monodromy assumptions. Moreover, the right-hand side of (\ref{last}) is an invariant of the  singularities of $X$, since the restrictions of $\Psi'^H_f([\Q^H_M])$ and $[IC'^H_X]$ over the regular part $X_{\rm reg}$ of $X$ coincide, so that
 $Gr^W_{n}\Psi^H_{f}(\Q_M^H[n+1])\ominus IC^H_X$ and $Gr^W_{k}\Psi^H_{f}(\Q_M^H[n+1])$ for $k\neq n$ are supported on $X_{\rm sing}$. 
 Therefore we get from (\ref{last}) as before (by the functoriality of ${MHT_{y}}_*$ for the 
 closed inclusion $X_{\rm sing}\hookrightarrow X$) a localized version
 \begin{equation}\label{last-loc} \begin{split}
(-1)^n \cdot {MHT_y}_*&\left([Gr^W_{n}\Psi^H_{f}(\Q_M^H[n+1])\ominus IC^H_X]+
 \sum_{k \neq n} [Gr^W_{k}\Psi^H_{f}(\Q_M^H[n+1])] \right)\\
 &=:\MC{IT_y}_*(X) \in H_*(X_{\rm sing})\otimes \Q[y] \: . \end{split}\end{equation} 
 
In particular, the classes ${T_y}^{\rm vir}_*(X)$ and ${IT_y}_*(X)$ coincide in degrees higher than the dimension of the singular locus. However, in general it is difficult to explicitly understand (\ref{last-loc}), except for simple situations. For example, if $X$ has only isolated singularities, the stalk calculation yields just as in Example \ref{ex:iso} that:
\begin{equation} \MC{IT_y}_*(X)=\sum_{x \in X_{\rm sing}} \left( \chi_y([H^*(F_x;\Q)]) - \chi_y([IH^*(c^{\circ}L_{x,X})] \right).
\end{equation}
And all special situations described earlier by examples have a counterpart in this case. We leave the details and precise formulations as an exercise for the interested reader.\\ 

More generally, under the hypotheses of Theorem \ref{main} and if $X$ is also reduced, (\ref{eq1}) and (\ref{eq70}) yield the following class formulae (which should be compared to the $L$-class formula (\ref{CS}) from the Introduction):
{\allowdisplaybreaks
\begin{eqnarray*} 
\MC{IT_y}_*(X)&:=&{T_y}^{\rm vir}_*(X) - {IT_y}_*(X)\\
&=&\sum_{V\in {\VV_0}} ({T_y}_*(\bar V) - {T_y}_*(\bar{V}\backslash V)) \cdot
(\chi_y([H^*(F_v;\Q)]) - \chi_y([IH^*(c^¡L_{V,X})]) \\
&=&\sum_{V \in \VV_0} \widehat{IT}_y(\bar V) \cdot \left( \chi_y([H^*(F_v;\Q)]) - \chi_y([IH^*(c^{\circ}L_{V,X})] \right)\:.
\end{eqnarray*}
}

\section{Geometric Consequences and Concluding Remarks}\label{fin}

As already pointed out in the Introduction, for the value $y=-1$ of the parameter, the Milnor-Hirzebruch class $\MC{T_y}_*(X)$ reduces to the rationalized Milnor class of $X$, which measures the difference between the Fulton-Johnson class \cite{FJ} and Chern-MacPherson class \cite{MP}.\\

Let us now consider the case when $y=0$. If the hypersurface $X$ has only {\it Du Bois singularities} (e.g., rational singularities, cf. \cite{Sa2}), then by (\ref{0}) we have that $\MC{T_0}_*(X)=0$, i.e., 
$${MHT_0}_*(\Phi'^H_f(\left[\Q^H_M\right]))=0 \in H_*(X)\otimes \Q \:.$$ 
In view of our main result, this vanishing (which is in fact a class version of Steenbrink's {\it cohomological insignificance of $X$} \cite{St}) imposes interesting geometric identities on the corresponding Todd-type invariants of the singular locus. For example, we obtain the following
\begin{cor}  If the hypersurface $X$ has only {\it isolated Du Bois singularities}, then
\begin{equation}\label{van} {\rm dim}_{\C} Gr^0_FH^n(F_x;\C)=0\end{equation}
for all  $x\in X_{\rm sing}$. 
\end{cor}
\noindent It should be pointed out that in this setting, by a result of Ishii \cite{I} one gets that (\ref{van}) is in fact equivalent to $x\in X_{\rm sing}$ being an isolated Du Bois hypersurface singularity.
Also note that in the arbitrary singularity case, the {\it Milnor-Todd class } $\MC{T_0}_*(X)$ carries interesting non-trivial information about the singularities of the hypersurface $X$.\\

Finally, if $y=1$, our main formula (\ref{eq1}) should be compared to the Cappell-Shaneson topological result of equation (\ref{CS}). While it can be shown (compare with \cite{Max05}) that the normal contribution $\sigma({\rm lk}(V))$ in (\ref{CS}) for a singular stratum $V \in \VV_0$ is in fact the signature $\sigma(F_v)$ ($v \in V$) of the Milnor fiber (as a manifold with boundary) of the singularity in a transversal slice to $V$, the precise relation
between $\sigma(F_v)$ and $\chi_1([\tilde{H}^*(F_v;\Q)])$ is in general very difficult to understand.
However,  in some cases it is possible to obtain such a ``local Hodge index theorem" (compare with equation (\ref{HIH}) for the global projective case):
\begin{prop}\label{lHi}
Assume the complex hypersurface $X=f^{-1}(0)$  is a  {\it rational homology manifold with only isolated singularities}. Then for any $x \in X_{\rm sing}$, we have:
\begin{equation}\label{fin} \sigma(F_x)=\chi_1([\tilde{H}^n(F_x;\Q)]).
\end{equation}
\end{prop} 
\begin{proof} If $X$ is of even
complex dimension $n$, the result follows form the following formula of Steenbrink  (see \cite{St5}[Thm.11]):
\begin{equation}\sigma(F_x)=\sum_{p+q=n} (-1)^p \left( h^{p,q}+ 2 \sum_{i\geq 1} (-1)^i h^{p+i,q+i}\right) \:,\end{equation}
with $h^{p,q}:={\rm dim} Gr^p_F Gr^W_{p+q} H^n(F_x;\C)$ the corresponding Hodge numbers of the mixed Hodge structure on $H^n(F_x;\Q)$.
Indeed, since $X$ is a rational homology manifold, we get by  \cite{St5}[p.293] that:
$$0={\rm dim} A^{p+i,q+i}_{n+2i}=h^{p+i,q+i} - h^{p-i,q-i} \:.$$
Moreover, the symmetry $h^{p,q}=h^{q,p}$ of the Hodge numbers under 
conjugation yields:
$$\sum_{p+q= odd} (-1)^p h^{p,q} = 0$$
Altogether, we get
\begin{equation}\sigma(F_x) = \sum_{p,q} (-1)^p h^{p,q} = \chi_1([\tilde{H}^n(F_x;\Q)]) \:.\end{equation}

In case $X$ is of odd complex dimension $n$, both terms of the claimed equality (\ref{fin}) vanish identically. Indeed, $\sigma(F_x)=0$ by definition, whereas  
the vanishing of $\chi_1([\tilde{H}^n(F_x;\Q)])$
follows from a duality argument, as in the proof of the classical Hodge
index theorem (\ref{HIH}). More precisely, one has a duality involution $\DD$ acting on
$K_0(\mh(-))$ and, resp.,  $K_0(var/-)[\Lef^{-1}]$ in a compatible way (e.g., see 
\cite{Sch3}[(47),(48)]), with $\DD$ the usual duality involution on
$K_0(\mh(pt))=K_0(\mhs^p)$. In particular,
$$\chi_y(\DD(-))= \chi_{1/y}(-) \quad \text{and} \quad
\chi_1(\DD(-))= \chi_1(-)$$
on $K_0(\mhs^p)$. Moreover,
$$\DD\Psi_f^m([id_M])= \Lef^{-n}\cdot \Psi_f^m([id_M])$$
(cf. \cite{Bit}[Thm.6.1]), and
$$\DD\circ \Psi^H_f (1) \simeq \Psi^H_f\circ \DD$$
on $D^b\mh(M)$ (cf. \cite{Sa1}[Prop.2.6]). Similarly,
$$\DD \Q^H_X \simeq \Q^H_X[2n](n) \:,$$
as $X$ is a rational homology manifold, so $\Q^H_X\simeq IC'^H_X$.
Putting this together, we get 
$$\DD [ \Phi^H_f( \Q^H_M ) ] = [\Phi^H_f( \Q^H_M ) (n)] \in K_0(\mh(X))\:.$$
Lastly, the isolated singularity $x\in X_{\rm sing}$ is an isolated point in the support of
$ \Phi^H_f( \Q^H_M )$ and $\DD  \Phi^H_f( \Q^H_M )$, respectively, thus
$$i_x^*\DD  \Phi^H_f( \Q^H_M ) \simeq i_x^! \DD  \Phi^H_f( \Q^H_M )
\simeq \DD i_x^*  \Phi^H_f( \Q^H_M ),$$
with $i_x: \{x\}\to X$  the inclusion map. We now get the desired vanishing 
$\chi_1([\tilde{H}^n(F_x;\Q)])=0$ from the following sequence of identities:
$$\chi_1(i_x^* [\Phi^H_f( \Q^H_M )]) = \chi_1(\DD i_x^* [\Phi^H_f( \Q^H_M )])
= \chi_1( i_x^*\DD [\Phi^H_f( \Q^H_M )])
= (-1)^n \chi_1( i_x^*[\Phi^H_f (\Q^H_M )]) \:.$$
\end{proof}

We can therefore prove in the setting of Prop.\ref{lHi} the following conjectural interpretation of $L$-classes from \cite{BSY}:
\begin{thm} Let $X$ be a compact complex algebraic variety with only isolated singularities,
which moreover is a  rational homology manifold and can be realized as a global hypersurface (of codimension one) in a complex algebraic manifold. Then \begin{equation}L_*(X)={IT_{y}}_*(X)\vert_{y=1}.\end{equation}
\end{thm}
\begin{proof} Assume the complex dimension $n$ of $X$ is even. Then, by combining (\ref{fin}),  (\ref{CS}) and (\ref{iso}) we get that
\begin{equation}\label{cp}L^{\rm vir}_*(X)-L_*(X)= \sum_{x\in X_{\rm sing}} \chi_1([\tilde{H}^n(F_x;\Q)]) \cdot [x] = {T_{1}}^{\rm vir}_*(X)-{IT_{1}}_*(X)\:.\end{equation}
For $n$ odd,  formula (\ref{cp}) is trivially true, as follows by the vanishing of the local signature and, resp., Hodge contributions at each of the singular points (cf. Prop.\ref{lHi}).

Next note that since $L^*(-)=T_y^*(-)\vert_{y=1}$ (cf. \cite{H}), we obtain an equality of the corresponding virtual classes, i.e., 
\begin{equation}\label{virt}L^{\rm vir}_*(X)={T_{1}}^{\rm vir}_*(X).\end{equation}
The result follows now from the identities (\ref{cp}) and (\ref{virt}).

\end{proof}

\begin{remark}\rm  
The conjectured equality $L_*(X)={T_1}_{*}(X)$ also holds in the case of a compact 
hypersurface $X$,
which is a rational homology manifold with $X_{\rm sing}$ smooth, so that
$X_{\rm sing}\subset X$ is a Whitney stratification with all components of
$X_{\rm sing}$ simply-connected. Indeed, this follows from  the arguments used in the above proof, applied to the Milnor fiber of a transversal slice to the singular locus, combined with the identities  (\ref{CS}) and (\ref{smm}).
\end{remark}

\providecommand{\bysame}{\leavevmode\hbox
to3em{\hrulefill}\thinspace}

\end{document}